\documentclass[11pt]{article}
\usepackage{graphicx}
\usepackage{color}
\usepackage{subcaption}
\usepackage{wrapfig}
\usepackage{float}
\usepackage{amsmath}
\usepackage{amssymb}
\usepackage{natbib}
\bibpunct[:]{(}{)}{,}{a}{}{,}
\usepackage{amsthm}
\usepackage{authblk}
\usepackage{authblk}
\usepackage[top=30truemm,bottom=30truemm,left=20truemm,right=20truemm]{geometry}
\usepackage[titletoc,toc,title]{appendix}
\usepackage{latexsym}
\usepackage{algorithm}
\usepackage{algorithmic}
\usepackage[english]{babel}
\newtheorem{thm}{Theorem}[section]
\newtheorem{lem}{Lemma}[section]
\newtheorem{prop}{Proposition}[section]

\newtheorem{cor}{Corollary}[section]
\providecommand{\keywords}[1]{\textbf{\textit{Keywords:}} #1}
\newcommand{\coshh}{\left(2\cosh \frac{p\theta}{2}\right)}
\newcommand{\sinhh}{\left(2\sinh \frac{p\theta}{2}\right)}

\newcommand{\malpha}[2]{
\raise1.5pt\hbox{$\displaystyle \mathop{\alpha}^{\rm m}$}\hspace{-1.5pt}_{#1}^{\, #2}}
\newcommand{\mH}[2]{
\raise1.5pt\hbox{$\displaystyle \mathop{H}^{\rm m}$}\hspace{-1.5pt}_{#1}^{\, #2}}
\newcommand{\mnabla}[2]{
\raise1.5pt\hbox{$\displaystyle \mathop{\nabla}^{\rm m}$}
\hspace{-1.5pt}_{#1}^{\, #2}}
\newcommand{\ealpha}[2]{
\raise1.5pt\hbox{$\displaystyle \mathop{\alpha}^{\rm e}$}
\hspace{-1.5pt}_{#1}^{\, #2}}
\newcommand{\enabla}[2]{
\raise1.5pt\hbox{$\displaystyle \mathop{\nabla}^{\rm e}$}\hspace{-1.5pt}_{#1}^{\, #2}}
\title{Shrinkage priors for circulant correlation structure models 
}
\date{}
\author[1]{Michiko Okudo}
\author[1]{Tomonari Sei}
\affil[1]{\small Department of Mathematical Informatics

Graduate School of Information Science and Technology

The University of Tokyo

7-3-1 Hongo, Bunkyo-ku, Tokyo 113-8656, JAPAN\\

Email: {okudo@mist.i.u-tokyo.ac.jp; sei@mist.i.u-tokyo.ac.jp}
}

\begin{document}
\maketitle

\begin{abstract}
We consider a new statistical model called the circulant correlation structure model, which is a multivariate Gaussian model with unknown covariance matrix and has a scale-invariance property.
We construct shrinkage priors for the circulant correlation structure models and show that Bayesian predictive densities based on those priors asymptotically dominate Bayesian predictive densities based on Jeffreys priors under the Kullback--Leibler (KL) risk function.
While shrinkage of eigenvalues of covariance matrices of Gaussian models has been successful, the proposed priors shrink a non-eigenvalue part of covariance matrices.
\end{abstract}

\vspace{0.5cm}

\keywords{exchangeable correlation structure, shrinkage priors, Bayes, multivariate analysis}

\section{Introduction}
\label{sec:intro}
We study shrinkage prediction of covariance matrices of multivariate Gaussian models.
Shrinkage estimation and prediction have been widely studied, originating from Stein's paradox and the James--Stein estimator.
Shrinkage of $\mu$ of $\mathrm{N}_p(\mu,I_p)$ to the origin is known to be effective, and shrinkage priors for Bayes estimation and prediction are also studied; see, for example, \cite{komaki2006} and \cite{george2012minimax}.
On the other hand, in shrinkage estimation of $\Sigma$ of $\mathrm{N}_p(0,\Sigma)$, eigenvalue shrinkage towards the origin has been successful and there has been extensive research including \cite{donoho2018}.
\cite{yang1994} investigated priors shrinking differences between eigenvalues.

In this paper, we propose a new model called circulant correlation structure model, which is a submodel of $\mathrm{N}_p (0,\Sigma)$, and show that shrinkage of non-eigenvalue part of $\Sigma$ has favorable effect in this model.
Our model is constructed as follows.
We consider $p$-dimensional Gaussian models $\mathrm{N}_p (0,\Sigma)$ with covariance matrices expressed as 
\begin{align}
\Sigma = D_\alpha R D_\alpha
\label{eq:model}
\end{align}
and
\begin{align}
R = QD_{\lambda}Q^*,
\label{eq:R}
\end{align}
where $D_\alpha={\rm diag}(\alpha_1,\dots,\alpha_p)$, 
$D_{\lambda} = \mathrm{diag}(\lambda_1,\ldots,\lambda_p)$
and $Q$ is a constant matrix expressed by
\[
 Q = \frac{1}{\sqrt{p}}\begin{pmatrix}
 1& 1& 1& \cdots& 1\\
 1& \omega& \omega^2& \cdots& \omega^{p-1}\\
 1& \omega^2& \omega^4& \cdots& \omega^{2(p-1)}\\
 \vdots& \vdots& \vdots& \ddots& \vdots\\
 1& \omega^{p-1}& \omega^{2(p-1)}& \cdots& \omega^{(p-1)^2}
 \end{pmatrix}.
\]
Here $\omega=\exp(2\pi\sqrt{-1}/p)$ is a primitive $p$-th root and $Q^*$ denotes the Hermitian transpose of $Q$.
The matrix $Q=(q_{ij})$ appears in discrete Fourier transformation and $Q$ is unitary.
The parameters $\alpha_i$ and $\lambda_k$ are assumed to be positive.

We call the model defined by (\ref{eq:model}) and (\ref{eq:R}) the circulant correlation structure model.
The structure of $R$ is sometimes called a circular model in time series analysis \citep[Section 6.5]{anderson1971}, where $\lambda$ determines the spectrum of $R$. Our model has an additional parameter vector $\alpha$, which plays a role of amplitude modulation \citep{jiang2004} and makes the model scale invariant.

The $(i,j)$ component of the matrix $R=(r_{ij})$ is
\begin{align}
 r_{ij} &= \sum_{k=1}^p \lambda_k q_{ik}\bar{q}_{jk}
 \nonumber \\
 &= \frac{1}{p}\sum_{k=1}^p \lambda_k \omega^{(i-j)(k-1)}
 \label{eq:cyclic}
\end{align}
where $\bar{q}_{ij}$ is the $(i,j)$ component of $Q^*$. 
The $(i,j)$ component of the matrix $R$ depends only on $i-j$ modulo $p$. 
Therefore, $R$ is a circulant matrix.
Conversely, $\lambda_k$ is determined from $R$ by
\begin{align}
    \lambda_k = \frac{1}{p}\sum_{i=1}^p\sum_{j=1}^p r_{ij}\bar\omega^{(i-j)(k-1)}.
    \label{eq:converse}
\end{align}
Because each component of $R$ is a real number, $\lambda_1,\ldots,\lambda_p$ must satisfy
\begin{align}
 \lambda_{a+1} = \lambda_{p-a+1},\quad 1\leq a\leq \lfloor (p-1)/2\rfloor.
 \label{eq:real}
\end{align}

To remove a multiplicative redundancy between $\alpha$ and $\lambda$, we assume
\begin{align}
 \prod_{i=1}^p \lambda_i = 1
 \label{eq:det-one}
\end{align}
without loss of generality.

\begin{lem}
 Under the restriction (\ref{eq:det-one}), the parameters $\alpha$ and $\lambda$ are identifiable from $\Sigma$.
\end{lem}

\begin{proof}
We see from the expression (\ref{eq:cyclic}) that all the diagonal elements of $R$ have the same value. This means that $R$ is a constant multiple of a correlation matrix determined from $\Sigma$. The multiplicative constant is obtained from an identity $\det(R)=\det(D_\lambda)=1$.
Then, $\alpha$ and $\lambda$ are determined by (\ref{eq:model}) and (\ref{eq:converse}), respectively.
\end{proof}

The problem settings are as follows.
We consider a statistical model
\[
\mathcal{P}=\{ \mathrm{N}_p (0,\Sigma)\mid \Sigma=D_\alpha R(\theta)D_\alpha,\ 
R(\theta)=QD_{\lambda(\theta)}Q^*,\ 
\theta\in\mathbb{R}^d,\ \alpha\in\mathbb{R}_+^p\}
\]
with parameters $\theta=(\theta_1,\dots,\theta_d)$ and $\alpha$, 
where $\lambda(\theta)$ is a parametric family satisfying (\ref{eq:real}) and (\ref{eq:det-one}).
Specific examples of $\lambda(\theta)$ are provided in Section~\ref{sec:shrinkage} and Section~\ref{sec:submodels}.
Suppose that we have observations $x^n = \{x(1),\dots,x(n)\}$ from $p(x;\theta,\alpha)\in\mathcal{P}$.
We address the problem of constructing a predictive density $\hat{p}(y\mid x^n)$ for a future sample $y$ that follows the same distribution $p(y;\theta,\alpha)$ as $x(i),~i=1,\dots,n$.
The performance of predictive density $\hat{p}(y\mid x^n)$ is evaluated by the Kullback--Leibler (KL) divergence
\[
D\{p(y;\theta,\alpha); \hat{p}(y\mid x^n)\} = \int p(y;\theta,\alpha) \log\frac{p(y;\theta,\alpha)}{\hat{p}(y\mid x^n)} dy,
\]
and the risk function and the Bayes risk functions are
\[
 {E}[D\{p(y;\theta,\alpha);\hat{p}(y\mid x^n) \}]=\int p(x^n;\theta,\alpha)D\{p(y;\theta,\alpha);\hat{p}(y\mid x^n) \}dx^n,
\]
and
\[
\int\pi(\theta,\alpha) \int p(x^n;\theta,\alpha)D\{p(y;\theta,\alpha);\hat{p}(y\mid x^n) \}dx^n d\theta d\alpha,
\]
respectively.
Bayesian predictive densities based on a prior $\pi(\theta,\alpha)$ for a future sample $y$ is obtained by
\[
p_\pi(y\mid x) = \int p(y;\theta,\alpha)\pi(\theta,\alpha \mid x^n) d\theta d\alpha
\]
where $\pi(\theta,\alpha \mid x^n)$ is the posterior density
\[\pi(\theta,\alpha \mid x^n) = \frac{p(x^n;\theta,\alpha)\pi(\theta,\alpha)}{\int p(x^n;\theta,\alpha)\pi(\theta,\alpha) d\theta d\alpha}. \]
It is shown in \cite{aitchison1975} that Bayesian predictive densities are optimal about the Bayes risk, thus we adopt them as predictive densities for each prior.

We propose shrinkage priors for the circulant correlation structure model, which shrinks the correlation part of covariance matrices to the identity matrix, and show that it dominates the Jeffreys prior asymptotically.
We consider shrinkage of $\theta$, which controls eigenvalues of $R$, not $\Sigma$.
The density of the Jeffreys prior of the model increases exponentially as $\theta$ increases.
We propose a uniform prior about $\theta$ and $\log\alpha_i~(i=1,\dots,p)$, and compared to the Jeffreys prior, it has shrinkage effect about $\theta$.
The model is intentionally designed so that the cross components of the Fisher information matrix about $\theta$ and $\alpha$ are 0, and it makes it easier to construct shrinkage prior of $\theta$ alone.
This property is analogous to the fact that in the full model $\{\mathrm{N}_p(0,\Sigma) \mid \Sigma\}$, the cross components of the Fisher information metric about eigenvalues and eigenvectors of $\Sigma$ are 0.

The construction of the rest of the paper is as follows.
In Section \ref{sec:prior}, we give the specific form of $\lambda(\theta)$, and propose a prior that dominates the Jeffreys prior asymptotically.
In Section \ref{sec:submodels}, optimality results for the proposed priors are given about the asymptotic KL risk for the full model and a submodel called exchangeable correlation structure model.
Numerical experiments to illustrates the difference of asymptotic KL risks of the Jeffreys prior and the proposed prior are also given. 

\section{Construction of shrinkage priors}
\label{sec:prior}
\label{sec:shrinkage}

In this section, we assume that $\lambda(\theta)$ is log-linear, that is, $\log\lambda(\theta)=(\log\lambda_k(\theta))_{k=1}^p$ is linear in $\theta=(\theta_1,\ldots,\theta_d)$. We also assume that vectors $(\partial/\partial\theta_i)\log\lambda$ ($i=1,\dots,d$) are linearly independent. For example, the full model for $\lambda$ that satisfies (\ref{eq:real}) and (\ref{eq:det-one}) is written in log-linear form as
\[
 \lambda(\theta) = (e^{-\theta_1-\dots-\theta_{p-1}},e^{\theta_1}\dots,e^{\theta_{p-1}}),
\]
where $d=\lfloor p/2\rfloor$ and $\theta_{p-a}=\theta_a$ for $1\leq a\leq \lfloor(p-1)/2\rfloor$.
Another example is the exchangeable correlation structure model $\lambda(\theta)=(e^{-(p-1)\theta},e^\theta,\dots,e^\theta)$ with $d=1$, which is discussed in Section~\ref{sec:submodels}.

We begin by evaluating the Fisher metric of the model.
Let $\bmod(a)\in\{1,\ldots,p\}$ denote the modulo of $a$ divided by $p$.
It is well known that the Fisher metric of $N(0,\Sigma(\omega))$ with respect to a parameter $\omega$ is
\[
g_{\omega_i\omega_j}=\frac{1}{2}\mathrm{tr}\left(
\Sigma^{-1}\frac{\partial\Sigma}{\partial\omega_i}
\Sigma^{-1}\frac{\partial\Sigma}{\partial\omega_j}
\right).
\]

\begin{lem}
Suppose that $\lambda(\theta)$ is log-linear.
The component of Fisher information matrix $g$ about $\theta_i$ and $\theta_j~~(i,j=1,\dots,d)$ is expressed as a constant matrix
\[
g_{\theta_i\theta_j} = \frac{1}{2}\sum_{k=1}^p\left(\frac{\partial}{\partial\theta_i}\log\lambda_k\right)\left(\frac{\partial}{\partial\theta_j}\log\lambda_k\right).
\]
The component of $g$ about $\alpha_i$ and $\alpha_j~~(i,j = 1,\dots,p)$ is
\begin{align*}
    g_{\alpha_i \alpha_j}  
    &= (\delta_{ij}+r_{ij}r^{ij})\alpha_i^{-1}\alpha_j^{-1}
    \\
    &=\left(\delta_{ij} + \sum_m q_{im}\bar{q}_{jm} \sum_{\bmod(k+l-1)=m}\frac{1}{p}\lambda_k\lambda_l^{-1}\right)\alpha^{-1}_i \alpha^{-1}_j.
\end{align*}
The other components $g_{\alpha_i \theta_j}~(i=1,\dots,p; j=1,\dots,d)$ are 0.
\end{lem}

\begin{proof}
The component of $g$ about $\theta$ is obtained by
\begin{align*}
    g_{\theta_i \theta_j}  &= \frac{1}{2}\mathrm{tr}\left(\Sigma^{-1} \frac{\partial \Sigma}{\partial \theta_i} \Sigma^{-1}\frac{\partial \Sigma}{\partial \theta_j}\right)\\
    &= \frac{1}{2}\mathrm{tr}\left(R^{-1} \frac{\partial R}{\partial \theta_i} R^{-1} \frac{\partial R}{\partial \theta_j}\right)\\
    &= \frac{1}{2}\mathrm{tr}\left(D_\lambda^{-1} \frac{\partial D_\lambda}{\partial \theta_i} D_\lambda^{-1} \frac{\partial D_\lambda}{\partial \theta_j}\right)\\
    &= \frac{1}{2}\sum_{k=1}^p\left(\frac{\partial}{\partial\theta_i}\log\lambda_k\right)\left(\frac{\partial}{\partial\theta_j}\log\lambda_k\right).
\end{align*}
In particular, $g_{\theta_i\theta_j}$ is a constant because $\lambda(\theta)$ is log-linear.
The component of $g$ about $\alpha$ is
\begin{align*}
    g_{\alpha_i \alpha_j}  &= \frac{1}{2}\mathrm{tr}\left(\Sigma^{-1} \frac{\partial \Sigma}{\partial \alpha_i} \Sigma^{-1}\frac{\partial \Sigma}{\partial \alpha_j}\right)\\
    &= \frac{1}{2}\mathrm{tr}\left(
    R^{-1} \left( D^{-1}\frac{\partial D_\alpha}{\partial \alpha_i} R + R D^{-1}\frac{\partial D_\alpha}{\partial \alpha_i}\right) R^{-1}\left( D^{-1}\frac{\partial D_\alpha}{\partial \alpha_j} R + R D^{-1}\frac{\partial D_\alpha}{\partial \alpha_j}\right)
    \right)\\
    &=(\delta_{ij} + r^{ij}r_{ij})\alpha^{-1}_i \alpha^{-1}_j,
\end{align*}
where $r_{ij}$ and $r^{ij}$ are the $(i,j)$ component of $R$ and $R^{-1}$, respectively.
Because
\[
 r_{ij} = \sum_{k=1}^p \lambda_k q_{ik}\bar{q}_{jk}
\]
and
\[
 r^{ij} = \sum_{k=1}^p \lambda_k^{-1} q_{ik}\bar{q}_{jk},
\]
we have
\begin{align*}
 r_{ij}r^{ij} &= \left(\sum_{k=1}^p\lambda_kq_{ik}\bar{q}_{jk}\right)
\left(\sum_{l=1}^p\lambda_l^{-1}q_{il}\bar{q}_{jl}\right)
\\
&= \sum_k \sum_l \lambda_k\lambda_l^{-1}q_{ik}\bar{q}_{jk}q_{il}\bar{q}_{jl}.
\end{align*}
We obtain
\begin{align*}
 r_{ij}r^{ij} 
 &= \sum_k \sum_l \lambda_k\lambda_l^{-1}q_{ik}\bar{q}_{jk}q_{il}\bar{q}_{jl}\\
 &= \sum_k \sum_l \lambda_k\lambda_l^{-1} \frac{1}{{p}}\omega^{(i-1)(k+l-2)} \frac{1}{{p}}\bar{\omega}^{(j-1)(k+l-2)}\\
 &= \frac{1}{p}\sum_{m=1}^p q_{im}\bar{q}_{jm} \sum_{\bmod(k+l-1)=m}\lambda_k\lambda_l^{-1},
\end{align*}
because $q_{ij} = \frac{1}{\sqrt{p}}\omega^{(i-1)(j-1)}.$

The other components of $g$ is 0, which is confirmed as follows. We have
\begin{align} \label{eq:cross}
    \frac{1}{2}\mathrm{tr} \left(\Sigma^{-1} \frac{\partial \Sigma}{\partial \alpha_i} \Sigma^{-1}\frac{\partial \Sigma}{\partial \theta_k}\right)
    &= \frac{1}{2}\mathrm{tr}\left(R^{-1} D_\alpha^{-1}\frac{\partial \Sigma}{\partial \alpha_i} D_\alpha^{-1} R^{-1} \frac{\partial R}{\partial \theta_k}\right)\nonumber \\
    &= \frac{1}{2}\mathrm{tr}\left(R^{-1} 
    \left(D_\alpha^{-1}\frac{\partial D_\alpha}{\partial \alpha_i}R + R\frac{\partial D_\alpha}{\partial \alpha_i} D_\alpha^{-1}\right)
    R^{-1} \frac{\partial R}{\partial \theta_k}\right) \nonumber \\
    &= \mathrm{tr}\left(\frac{\partial D_\alpha}{\partial\alpha_i}D_\alpha^{-1}R^{-1}\frac{\partial R}{\partial\theta_k}\right)
\end{align}
and when all the diagonal components of
\[R^{-1} \frac{\partial R}{\partial \theta_k} \]
are 0, (\ref{eq:cross}) is 0.
The $i$-th diagonal component of $R^{-1} ({\partial R}/{\partial \theta_k})$ is
\begin{align*}
    \sum_j |q_{ij}|^2 \frac{\partial \log \lambda_j}{\partial \theta_k}
    = \frac{1}{p}\frac{\partial}{\partial\theta_k}\left(\sum_j\log\lambda_j\right) =0
\end{align*}
due to $|q_{ij}|^2=1/p$ and the condition (\ref{eq:det-one}).
\end{proof}

Because $g_{\alpha_i \theta_j}=0~~(i=1,\dots,p; j=1,\dots,d)$, the parameters $\theta$ and $\alpha$ are orthogonal with respect to the Fisher metric.
This property is called adaptivity if either $\alpha$ or $\theta$ is considered as a nuisance parameter; see e.g.\ \cite{segers2014}.

The density of the Jeffreys prior $\pi_J$ is evaluated as
\begin{align*}
    \pi_J(\theta,\alpha)
    &= |g|^{1/2}\\
    &= (|g_\theta| |g_\alpha|)^{1/2}\\
    &\propto |g_\alpha|^{1/2},
\end{align*}
where $g_\alpha=(g_{\alpha_i\alpha_j})$, and we used the fact that $g_\theta=(g_{\theta_i\theta_j})$ is constant.
Let
\[ \mu_m = \sum_{\bmod(k+l-1)=m}\frac{1}{p}\lambda_k\lambda_l^{-1},\]
then
\begin{align} \label{eq:logconv}
 |I+R\circ R^{-1}|
 &=\left|I + Q{\rm diag}(\mu_1,\ldots,\mu_m)Q^*\right| \nonumber
 \\
 &= \prod_{m=1}^p (1+\mu_m) \nonumber
 \\
 &= \prod_{m=1}^p \left(1+\sum_{\bmod(k+l-1)=m}\frac{1}{p}\lambda_k\lambda_l^{-1}\right),
\end{align}
where $A\circ B$ is the Hadamard product of $A$ and $B$.
We consider a new parametrization
\[
\beta_i = \log \alpha_i~~(i=1,\dots,p)
\]
 so that $g_{\beta_i \beta_j} = (\delta_{ij}+r^{ij}r_{ij}).$
Then, from (\ref{eq:logconv}), 
\begin{align}
\label{eq:gbeta}
|g_\beta|
 &=|I+R\circ R^{-1}| \nonumber \\
 &= \prod_{m=1}^{p} \left(1+\sum_{\bmod(k+l-1)=m}\frac{1}{p}\lambda_k\lambda_l^{-1}\right)
\end{align}
and
\begin{align*}
    \pi_J(\theta, \beta)= |g_\beta|^{1/2}.
\end{align*}
In particular, $\pi_J(\theta,\beta)$ does not depend on $\beta$.

We obtain a superharmonic prior for this model.
\begin{thm}
Suppose that $\lambda(\theta)$ is log-linear.
    Let
    \[
    \pi_S(\theta,\beta) \propto 1.
    \]
    The Bayesian predictive density based on $\pi_S$ asymptotically dominates the Bayesian predictive density based on $\pi_J$ regarding the KL risk.
\end{thm}

\begin{proof}
From the result in \cite{komaki2006}, it is enough to show that $\pi_S/\pi_J$ satisfies
\[
\Delta \left( \frac{\pi_S}{\pi_J}\right) < 0.
\]
Here, $\Delta$ denotes the Laplace--Beltrami operator
\begin{align*}
    \Delta f = \sum_{a,b}|g|^{-1/2}\partial_a (|g|^{1/2}g^{ab}\partial_b f),
\end{align*}
where $\partial_a = \partial/\partial \omega_a$ and $\omega=(\beta,\theta)$.
We show the following lemma.

\begin{lem}
For any function $f(\theta)$ depending only on $\theta$, we have
\begin{align}
\label{eq:finvf}
     f^{-1}\Delta f
    = \sum_{i,j} g^{\theta_i\theta_j} \left\{ \partial_{\theta_i}\partial_{\theta_j}\log f + (\partial_{\theta_i}\log f)(\partial_{\theta_j}\log f) + \partial_{\theta_i}\log|g|^{1/2}\partial_{\theta_j} \log f \right\}.
\end{align}
\end{lem}

\begin{proof}
Because $g^{\beta_i \theta_j}$ are 0 and $g^{\theta_i \theta_j}$ are constant, we have
\begin{align*}
    \Delta f &= \sum_{i,j} |g|^{-1/2}\partial_{\theta_i} (|g|^{1/2}g^{\theta_i\theta_j}\partial_{\theta_j} f) \nonumber \\
    &= \sum_{i,j} g^{\theta_i\theta_j}|g|^{-1/2}\partial_{\theta_i} (|g|^{1/2}\partial_{\theta_j} f) \nonumber \\
    &= \sum_{i,j} g^{\theta_i\theta_j} \left\{ \partial_{\theta_i}\partial_{\theta_j}f +  \partial_{\theta_i}\log|g|^{1/2}\partial_{\theta_j} f \right\}
\end{align*}
and we immediately obtain (\ref{eq:finvf}).
\end{proof}

To complete the proof of the theorem,
we show that the function $f=(\pi_S/\pi_J)$ satisfies $\Delta f<0$.
For
\[ f=\pi_J^{-1} \propto |g_\beta|^{-1/2}, \]
(\ref{eq:finvf}) is
\begin{align} \label{eq:ffinv}
    f^{-1}\Delta f
    &= -\frac{1}{2}\sum_{i,j} g^{\theta_i\theta_j} \partial_{\theta_i}\partial_{\theta_j}\log |g_\beta|.
\end{align}
We show that the Hessian matrix of $\log |g_\beta|$ is positive definite.
From (\ref{eq:gbeta}), 
\begin{align*}
 |g_\beta| =\prod_{m=1}^{p} \left(1+\sum_{\bmod(k+l-1)=m}\frac{1}{p}\lambda_k\lambda_l^{-1}\right).
\end{align*}
The function
\begin{align*}
    h(x_1,\dots,x_p) = \log\left(1 + \sum_{i=1}^p \exp(x_i)\right)
\end{align*}
is convex, thus its Hessian matrix $H_h$ is positive definite.
The Hessian matrix of the function
\[
(\theta_1,\dots,\theta_d)\mapsto \log \left(1+\sum_{\bmod(k+l-1)=m} \frac{1}{p}\lambda_k\lambda_l^{-1}\right)
\]
for a fixed $m$ is decomposed as
\[
W_m^\top H_h W_m,
\]
where $W_m\in\mathbb{R}^{p\times d}$ is a constant matrix defined by
\[
(W_m)_{li} = \frac{\partial}{\partial{\theta_i}}\log\left(\frac{1}{p}\lambda_k\lambda_l^{-1}\right),\quad k=\bmod(m+1-l).
\]
Since $W_m^\top H_h W_m$ is positive semi-definite, the Hessian matrix of $\log |g_\beta|$ is also positive semi-definite.
To see positive definiteness of the Hessian matrix, suppose that there exists $v\in\mathbb{R}^d$ such that $v^\top W_m^\top H_hW_mv=0$ for all $m$. This implies $W_mv=0$ by positive definiteness of $H_h$. From the expression of $W_m$, we obtain
\[
\sum_i v_i\frac{\partial}{\partial\theta_i}\log\lambda_k=\sum_i v_i\frac{\partial}{\partial\theta_i}\log\lambda_l
\]
for all $k$ and $l$. From the condition (\ref{eq:det-one}), we further obtain
\[
\sum_iv_i\frac{\partial}{\partial\theta_i}\log\lambda_k=0.
\]
Then $v=0$ follows from the linear independence of $(\partial/\partial\theta_i)\log\lambda$.
Hence, the Hessian matrix of $\log|g_\beta|$ is positive definite.
From (\ref{eq:ffinv}), we conclude that
$f^{-1}\Delta f < 0$.
\end{proof}

\section{Optimality in a class of priors}
\label{sec:submodels}

In this section, we consider the full model and exchangeable model as specific examples of the circulant correlation structure models. The superharmonic prior derived in Section~\ref{sec:shrinkage} has an optimal property in each case.

\subsection{Full model}

We recall the definition of the full model:
\begin{align}
 \lambda(\theta) = (e^{-\theta_1-\dots-\theta_{p-1}},e^{\theta_1},\dots,e^{\theta_{p-1}}),
 \quad \theta_a=\theta_{p-a},\quad 1\leq a\leq \lfloor(p-1)/2\rfloor
 \label{eq:full}
\end{align}
with $d=\lfloor p/2\rfloor$.
The prior $\pi_S$ is the uniform density with respect to the parameter $(\theta,\beta)$, where $\beta_i=\log\alpha_i$.
If $p=2$, the prior $\pi_S$ coincides with the correlation-shrinkage prior proposed by \cite{sei2022}, where finite-sample properties of $\pi_S$ are studied.

The prior $\pi_S$ has an optimal property regarding the asymptotic KL risk in a class of priors $\pi_a(\theta,\beta) \propto \pi_J(\theta,\beta)^{a}~(a\in\mathbb{R})$.

\begin{thm} \label{thm:optim-full}
Suppose that $\lambda(\theta)$ is the full model (\ref{eq:full}).
    Let $\pi_a(\theta,\beta) \propto \pi_J(\theta,\beta)^{a}~(a\in\mathbb{R}).$
    Regarding the asymptotic KL risk, when $\theta_1,\dots,\theta_d\to\infty$, $a=0$ is optimal, that is the prior $\pi_S(\theta,\beta)\propto 1$ is optimal.
    Here, $\theta_1,\ldots,\theta_d\to\infty$ means that $\theta_i=r\theta_{i0}$ for a fixed positive vector $\theta_{i0}$ and $r\to\infty$.
\end{thm}
\begin{proof}
    In order to simplify calculation, we consider a class of priors $\pi_t(\theta,\beta) \propto \pi_J(\theta,\beta)^{2t+1}~(t\in\mathbb{R}).$
    Let $\hat{p}_J$ and $\hat{p}_t$ be Bayesian predictive densities based on $\pi_J$ and $\pi_t$, respectively.
    The asymptotic risk difference of $\hat{p}_J$ and $\hat{p}_t$ is obtained following \cite{komaki2006}:
\begin{align}
\label{eq:riskdif}
{E}[D(p(y;\theta,\beta);\hat{p}_J )]-
 {E}[D(p(y;\theta,\beta);\hat{p}_t )]
= -\frac{2}{n^2}\left(\frac{\pi_J}{\pi_t}\right)^{1/2}\Delta\left(\frac{\pi_t}{\pi_J} \right)^{1/2} + \mathrm{o}(n^{-2}).
\end{align}
Let $f = ({\pi_t}/{\pi_J})^{1/2} = {\pi_J}^t = |g|^{t/2}$, then
\begin{align*} 
 \left(\frac{\pi_J}{\pi_t}\right)^{1/2}\Delta\left(\frac{\pi_t}{\pi_J} \right)^{1/2}
 &=\sum_{i,j} g^{\theta_i\theta_j} \left\{ \partial_{\theta_i}\partial_{\theta_j}\log f + (\partial_{\theta_i}\log f)(\partial_{\theta_j}\log f) + \partial_{\theta_i}\log|g|^{1/2}\partial_{\theta_j} \log f \right\}\nonumber \\
 &=\sum_{i,j} g^{\theta_i\theta_j} \left\{ \partial_{\theta_i}\partial_{\theta_j}\log f
 + \frac{1}{4}(t^2+t)\partial_{\theta_i}\log|g|\partial_{\theta_j} \log |g| \right\}.
\end{align*}
Let $S_m = \{(k,l) \mid \bmod(k+l-1)=m \}$.
We have 
\begin{align*}
    \partial_{\theta_i}\partial_{\theta_j}\log f 
    &= \frac{t}{2}\sum_{m=1}^{p}\partial_{\theta_i}\partial_{\theta_j}\log \left(1+\sum_{(k,l)\in S_m} \frac{1}{p}\lambda_k\lambda_l^{-1}\right)\\
    &=  \frac{t}{2}\sum_{m=1}^{p} \left\{
    \frac{\partial_{\theta_i}\partial_{\theta_j}\sum_{(k,l)\in S_m} \frac{1}{p}\lambda_k\lambda_l^{-1}}{1+\sum_{(k,l)\in S_m} \frac{1}{p}\lambda_k\lambda_l^{-1}}\right.\\
    &~~~\left. -\frac{(\partial_{\theta_i}\sum_{(k,l)\in S_m} \frac{1}{p}\lambda_k\lambda_l^{-1})(\partial_{\theta_j}\sum_{(k,l)\in S_m} \frac{1}{p}\lambda_k\lambda_l^{-1})}{(1+\sum_{(k,l)\in S_m} \frac{1}{p}\lambda_k\lambda_l^{-1})^2}
    \right\}.
\end{align*}
We prove that $\partial_{\theta_i}\partial_{\theta_j}\log f$ converges to 0 as $\theta_1,\ldots,\theta_d\to \infty$. Then,
\begin{align*}
    &\lim_{\theta_1,\dots,\theta_d\to\infty}\partial_{\theta_i}\partial_{\theta_j}\log f\\
    &= \lim_{\theta_1,\dots,\theta_d\to\infty}
    \frac{t}{2}\sum_{m=1}^{p} \left\{
    \frac{\partial_{\theta_i}\partial_{\theta_j}\sum_{(k,l)\in S_m} \frac{1}{p}\lambda_k\lambda_l^{-1}}{1+\sum_{(k,l)\in S_m} \frac{1}{p}\lambda_k\lambda_l^{-1}}\right. \\
    &~~~\left. -\frac{(\partial_{\theta_i}\sum_{(k,l)\in S_m} \frac{1}{p}\lambda_k\lambda_l^{-1})(\partial_{\theta_j}\sum_{(k,l)\in S_m} \frac{1}{p}\lambda_k\lambda_l^{-1})}{(1+\sum_{(k,l)\in S_m} \frac{1}{p}\lambda_k\lambda_l^{-1})^2}
    \right\}.
\end{align*}
Let $c_i^k = \partial_{\theta_i}\log\lambda_k$.
Then,
\begin{align*}
    &\lim_{\theta_1,\dots,\theta_d\to\infty}
    \sum_{m=1}^{p} \frac{\partial_{\theta_i}\partial_{\theta_j}\sum_{(k,l)\in S_m} \frac{1}{p}\lambda_k\lambda_l^{-1}}{1+\sum_{(k,l)\in S_m} \frac{1}{p}\lambda_k\lambda_l^{-1}} \\
    &=\lim_{\theta_1,\dots,\theta_d\to\infty}
    \sum_{m=2}^{p} \frac{\partial_{\theta_i}\partial_{\theta_j}\sum_{(k,l)\in S_m} \frac{1}{p}\lambda_k\lambda_l^{-1}}{1+\sum_{(k,l)\in S_m} \frac{1}{p}\lambda_k\lambda_l^{-1}} \\ 
    &= \lim_{\theta_1,\dots,\theta_d\to\infty}
    \sum_{m=2}^{p} \frac{\frac{1}{p}\sum_{\substack{(k,l)\in S_m}}(c^k_i-c^l_i)(c^k_j-c^l_j)\lambda_k\lambda_l^{-1}}
    {1+\sum_{(k,l)\in S_m} \frac{1}{p}\lambda_k\lambda_l^{-1}} \\ 
    &= \lim_{\theta_1,\dots,\theta_d\to\infty}
    \sum_{m=2}^{p} \frac{\frac{1}{p} 
    (c_i^m-c_i^1)(c_j^m-c_j^1)\lambda_1^{-1}\lambda_m}
    {1+\sum_{(k,l)\in S_m} \frac{1}{p}\lambda_k\lambda_l^{-1}} \\
    &= \sum_{m=2}^{p}(c_i^m-c_i^1)(c_j^m-c_j^1),
\end{align*}
where the first equality follows from $\lambda_k=\lambda_l$ for all $(k,l)\in S_1$ due to (\ref{eq:real}) and the third equality holds because $\lambda_1^{-1}\lambda_m$ is dominant as $\theta_1,\ldots,\theta_d\to\infty$.
Also, because
\begin{align}
\label{eq:dainikou}
    &\lim_{\theta_1,\dots,\theta_d\to\infty}\sum_{m=1}^{p}\frac{(\partial_{\theta_i}\sum_{(k,l)\in S_m} \frac{1}{p}\lambda_k\lambda_l^{-1})(\partial_{\theta_j}\sum_{(k,l)\in S_m} \frac{1}{p}\lambda_k\lambda_l^{-1})}{(1+\sum_{(k,l)\in S_m} \frac{1}{p}\lambda_k\lambda_l^{-1})^2}\nonumber\\
    &=\lim_{\theta_1,\dots,\theta_d\to\infty}\sum_{m=2}^{p}\frac{(\partial_{\theta_i}\sum_{(k,l)\in S_m} \frac{1}{p}\lambda_k\lambda_l^{-1})(\partial_{\theta_j}\sum_{(k,l)\in S_m} \frac{1}{p}\lambda_k\lambda_l^{-1})}{(1+\sum_{(k,l)\in S_m} \frac{1}{p}\lambda_k\lambda_l^{-1})^2}\nonumber\\
    &=\sum_{m=2}^p(c_i^m-c_i^1)(c_j^m-c_j^1), \nonumber
\end{align}
we have
\begin{align*}
    \lim_{\theta_1,\dots,\theta_d\to\infty}\partial_{\theta_i}\partial_{\theta_j}\log f = 0.
\end{align*}
Therefore, from (\ref{eq:riskdif})
\begin{align*}
    &\lim_{\theta_1,\dots,\theta_d\to\infty}{E}[D(p(y;\theta,\beta);\hat{p}_J )]-{E}[D(p(y;\theta,\beta);\hat{p}_t )]\\
    &= -\frac{2}{n^2}\lim_{\theta_1,\dots,\theta_d\to\infty}\left(\frac{\pi_J}{\pi_t}\right)^{1/2}\Delta\left(\frac{\pi_t}{\pi_J} \right)^{1/2} + \mathrm{o}(n^{-2})\\
    &= -\frac{2}{n^2}\lim_{\theta_1,\dots,\theta_d\to\infty}
    \sum_{i,j} g^{\theta_i\theta_j} \left\{ \partial_{\theta_i}\partial_{\theta_j}\log f
     + \frac{1}{4}(t^2+t)\partial_{\theta_i}\log|g|\partial_{\theta_j} \log |g| \right\}+ \mathrm{o}(n^{-2})\\
     &= -\frac{2}{n^2}\lim_{\theta_1,\dots,\theta_d\to\infty}
    \sum_{i,j} g^{\theta_i\theta_j} \left\{ \frac{1}{4}\left\{\left(t+\frac{1}{2}\right)^2-\frac{1}{4}\right\}\partial_{\theta_i}\log|g|\partial_{\theta_j} \log |g| \right\}+ \mathrm{o}(n^{-2})
\end{align*}
and it is maximized when $t=-1/2$.
\end{proof}
Theorem~\ref{thm:optim-full} assumes that $\theta_1,\ldots,\theta_d$ are positive. The same consequence holds for other cases whenever only a term in $\sum_{(k,l)\in S_m}\lambda_k\lambda_l^{-1}$ is dominant for each $m$.

\subsection{Exchangeable model}

If $\lambda(\theta)=(e^{-(p-1)\theta},e^\theta,\dots,e^\theta)$, the matrix $R(\theta)$ has exchangeable correlations. Indeed, we have
\begin{align*}
 r_{ij} &= \sum_{k=1}^p q_{ik}\lambda_k\bar{q}_{jk}
 \\
 &=e^{-(p-1)\theta}q_{i1}\bar{q}_{j1} + e^\theta\sum_{k=2}^p q_{ik}\bar{q}_{jk}
 \\
 &= \frac{1}{p}\left(
 e^{-(p-1)\theta} + e^\theta \sum_{k=2}^p \omega^{(i-j)(k-1)}
 \right)
 \\
 &= \frac{1}{p}\left(
 e^{-(p-1)\theta} + e^\theta (p\delta_{ij}-1)
 \right)
\end{align*}
since $\sum_{k=2}^p\omega^{(i-j)(k-1)}=-1$ if $i\neq j$.
The $p\times p$ submatrix $g_\alpha$ and the component regarding $\theta$ of its Fisher information matrix are
\begin{align*}
        g_\alpha &= -\frac{1}{p^2}\left(2\sinh \frac{p\theta}{2}\right)^2 \alpha^{-1}(\alpha^{-1})^\top 
    + \left(\frac{1}{p}\left(2\sinh \frac{p\theta}{2}\right)^2+2\right)\mathrm{diag}(\alpha_1^{-2},\dots, \alpha_p^{-2}),\\
    g_{\theta\theta}&=p(p-1)/2,
\end{align*}
respectively.
The Jeffreys prior is 
\begin{align*}
    \pi_J(\theta,\alpha) \propto \left( \frac{1}{p} \left(2\sinh \frac{p\theta}{2}\right)^2 + 2\right)^{ \frac{p-1}{2} }\left( \prod_{i=1}^p \alpha_i^{-1} \right).
\end{align*}

We have some theorems about shrinkage priors in the settings. Recall that $\beta_i=\log\alpha_i$.
\begin{prop}
Let $\pi_c(\theta, \beta) \propto \pi_J(\theta, \beta)^c~(c\in\mathbb{R})$.
The Bayesian predictive density based on $\pi$ asymptotically dominates the Bayesian predictive density based on $\pi_J$ regarding the KL risk when $-1\leq c < 1$.
\end{prop}

\begin{proof}
Let
$f=(\pi_c/\pi_J)^{1/2}=|g|^\gamma$, where $\gamma=(c-1)/4$.
We show $f$ is superharmonic when $-1/2\leq\gamma<0.$
We have
\[
 f^{-1}\Delta f = \frac{2}{p(p-1)}\left\{\gamma (\partial_\theta^2\log|g_\alpha|) + (\gamma^2+\frac{\gamma}{2})(\partial_\theta \log|g_\alpha|)^2\right\}.
\]
 Because $\log|g_\alpha|$ can be expressed as $\log|g_\alpha|=(p-1)\log(ae^{p\theta}+be^{-p\theta}+c) + (\rm{terms~independent~of~}\theta)$ using $a,b,c>0$, $\log|g_\alpha|$ is convex as a function of $\theta$ and we have $f^{-1}\Delta f\leq 0$ when $-1/2\leq\gamma<0.$     
\end{proof}

\begin{cor}
\label{cor:optimal}
    Let $\pi_c(\theta,\beta) \propto \pi_J(\theta,\beta)^{c}~(c\in\mathbb{R}).$
    Regarding the asymptotic KL risk, when $\theta \to\infty$, $c=0$ is optimal, that is the prior $\pi_S(\theta,\beta)\propto 1$ is optimal.
\end{cor}
\begin{proof}
Let
$f=(\pi_c/\pi_J)^{1/2}=|g|^\gamma$, where $\gamma=(c-1)/4$.
We have
\begin{align}
 f^{-1}\Delta f = \frac{2}{p(p-1)}\left\{\gamma (\partial_\theta^2\log|g_\alpha|) + (\gamma^2+\frac{\gamma}{2})(\partial_\theta \log|g_\alpha|)^2\right\}.    
\label{eq:cor31}
\end{align}
Let $S=\sinh(p\theta/2)$ and $C=\cosh(p\theta/2)$.
Here,
\begin{align}
\label{eq:partial_log1}
    \partial_\theta \log|g| &= (p-1)\frac{4SC}{\frac{4}{p}S^2 + 2},\\
\label{eq:partial_log2}
    \partial_\theta\partial_\theta \log|g| &=
    \frac{2p(p-1)}{(\frac{4}{p}S^2 + 2)^2}\left( (C^2+S^2)(\frac{4}{p}S^2 + 2) -\frac{8}{p}S^2C^2 \right)
\end{align}
and as $\theta \to \infty$
\begin{align*}
    \partial_\theta \log|g| &\to p(p-1),\\
    \partial_\theta^2 \log|g| &\to 0.
\end{align*}
Thus (\ref{eq:cor31}) is maximized by $\gamma = -1/4$.
\end{proof}

We plot the difference of asymptotic KL risks of the Bayesian predictive densities of the Jeffreys prior $\pi_J$ and the proposed prior for exchangeable correlation structure models.
Let $\hat{p}_{\gamma}$ be the Bayesian predictive density based on the prior $\pi_{\gamma}(\theta,\beta)\propto \pi_J(\theta,\beta)^{4\gamma + 1}$.
The asymptotic risk difference between $\pi_J$ and $\pi_S$ in this model is evaluated from (\ref{eq:riskdif}), (\ref{eq:cor31}), (\ref{eq:partial_log1}), and (\ref{eq:partial_log2}) as
\begin{align}
\label{eq:plot_difference}
&{E}[D(p(y;\theta,\beta);\hat{p}_J )]-
 {E}[D(p(y;\theta,\beta);\hat{p}_{\gamma} )] \nonumber\\
&= -\frac{2}{n^2}\left(\frac{\pi_J}{\pi_S}\right)^{1/2}\Delta\left(\frac{\pi_S}{\pi_J} \right)^{1/2} + \mathrm{o}(n^{-2})\nonumber\\
&= -\frac{2}{n^2}\frac{2}{p(p-1)}\left\{\gamma (\partial_\theta^2\log|g_\alpha|) + (\gamma^2 + \frac{\gamma}{2})(\partial_\theta \log|g_\alpha|)^2\right\}+ \mathrm{o}(n^{-2})\nonumber\\
&= -\frac{2}{n^2}\frac{2}{p(p-1)}\left\{\gamma\frac{2p(p-1)}{(\frac{4}{p}S^2 + 2)^2}\left( (C^2+S^2)(\frac{4}{p}S^2 + 2) -\frac{8}{p}S^2C^2 \right) + (\gamma^2 + \frac{\gamma}{2})\left((p-1)\frac{4SC}{\frac{4}{p}S^2 + 2}\right)^2\right\}+ \mathrm{o}(n^{-2}).
\end{align}

The asymptotic KL risk difference (\ref{eq:plot_difference}) is shown in Figure \ref{fig:KLrisk}.
The sample size $n$ is $n=100$ in every plot.
The risk difference ${E}[D(p(y;\theta,\beta);\hat{p}_J )]-
 {E}[D(p(y;\theta,\beta);\hat{p}_\gamma )]$ is positive, which shows that $\pi_\gamma$ dominates $\pi_J$ asymptotically.
 As shown in Corollary \ref{cor:optimal}, $\gamma=-1/4$ is optimal when $\theta\to\infty$, and the risk differences when $\gamma=-1/4$, $p=2,3,10$ are positive even in the area where $\theta$ is not around 0. 
 When $p=2$, the risk difference is large around $\theta=0$, and it means that proposed prior decreases the risk effectively around $D_\lambda=I_2$.
 On the other hand, when $p=3, 10$ and $\gamma = -1/4, -1/100$, the risk difference around $\theta=0$ is smaller, and it shows the possibility that stronger shrinkage around $\theta=0$ is effective.  

\begin{figure}[htbp]
    \centering
    \includegraphics[width=0.48\textwidth]{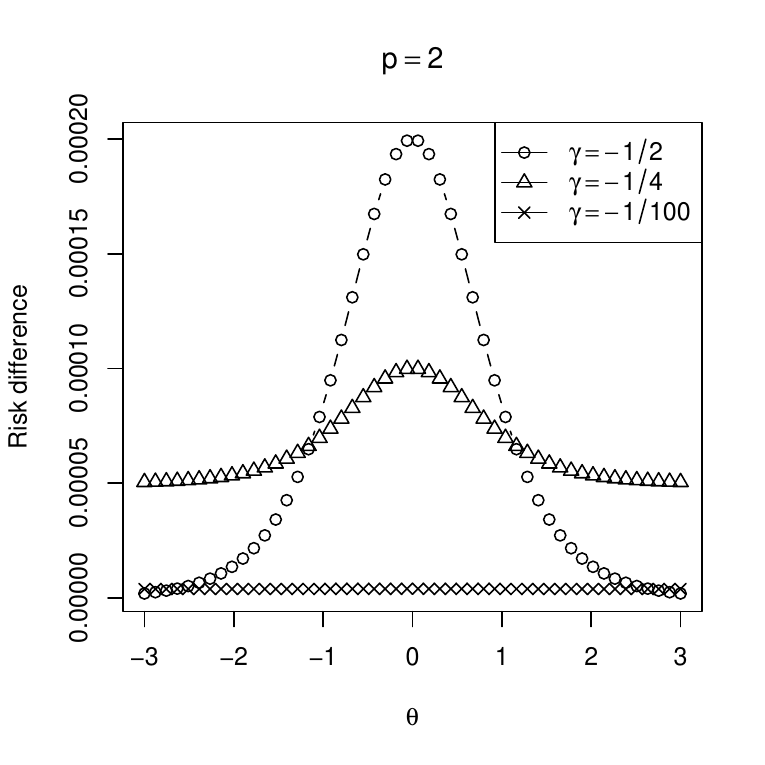}
    \includegraphics[width=0.48\textwidth]{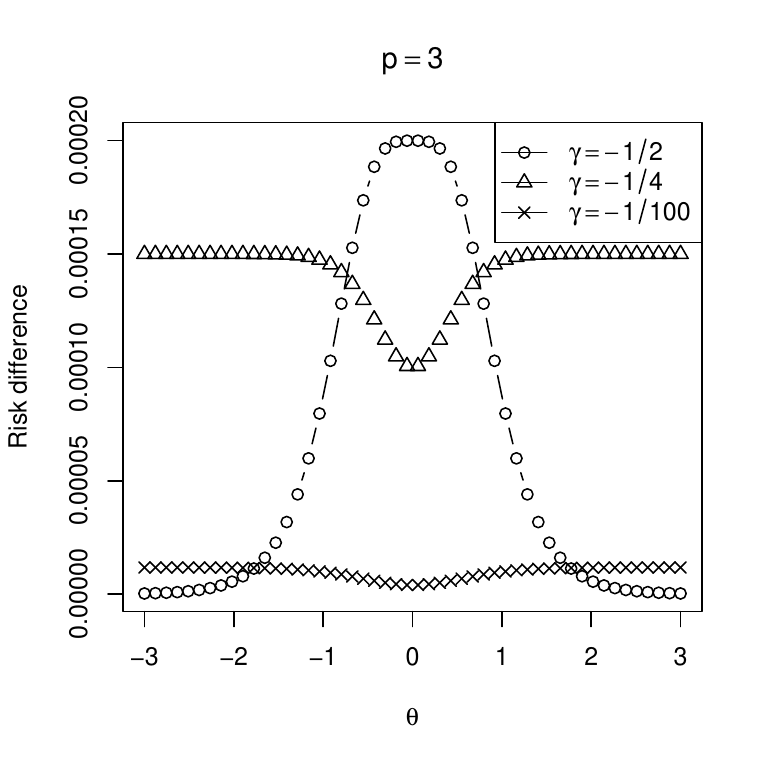}
    \includegraphics[width=0.48\textwidth]{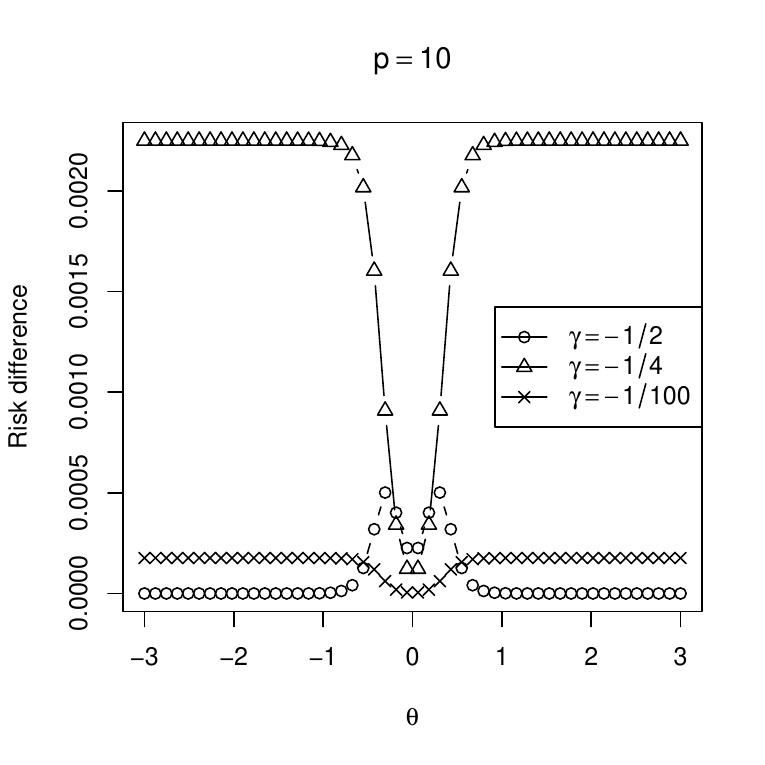}
    \caption{KL risk difference ${E}[D(p(y;\theta,\beta);\hat{p}_J )]-
 {E}[D(p(y;\theta,\beta);\hat{p}_\gamma )]$ for $p=2,3,10$ and $\gamma=-1/2,-1/4,-1/100$
 }
    \label{fig:KLrisk}
\end{figure}

\section{Acknowledgements}
The authors thank Fumiyasu Komaki for helpful comments to the early version of this work. 
This work was supported in part by JSPS KAKENHI Grant Numbers JP20K23316, JP21K11781 and JP22H00510.

\bibliographystyle{plainnat}
\bibliography{circulant}

\newcommand{\noop}[1]{}
\begin{thebibliography}{9}
\providecommand{\natexlab}[1]{#1}
\providecommand{\url}[1]{\texttt{#1}}
\expandafter\ifx\csname urlstyle\endcsname\relax
  \providecommand{\doi}[1]{doi: #1}\else
  \providecommand{\doi}{doi: \begingroup \urlstyle{rm}\Url}\fi

\bibitem[Aitchison(1975)]{aitchison1975}
J.~Aitchison.
\newblock Goodness of prediction fit.
\newblock \emph{Biometrika}, 62:\penalty0 547--554, 1975.

\bibitem[Anderson(1971)]{anderson1971}
T.~W. Anderson.
\newblock \emph{The Statistical Analysis of Time Series}.
\newblock Wiley, New York, 1971.

\bibitem[Donoho et~al.(2018)Donoho, Gavish, and Johnstone]{donoho2018}
D.~L. Donoho, M.~Gavish, and I.~M. Johnstone.
\newblock Optimal shrinkage of eigenvalues in the spiked covariance model.
\newblock \emph{{Annals of Statistics}}, 46:\penalty0 1742--1778, 2018.

\bibitem[George et~al.(2012)George, Liang, and Xu]{george2012minimax}
E.~I. George, F.~Liang, and X.~Xu.
\newblock From minimax shrinkage estimation to minimax shrinkage prediction.
\newblock \emph{Statistical Science}, 27:\penalty0 92--94, 2012.

\bibitem[Jiang and Hui(2004)]{jiang2004}
J.~Jiang and Y.~V. Hui.
\newblock Spectral density estimation with amplitude modulation and outlier detection.
\newblock \emph{Annals of the Institute of Statistical Mathematics}, 56:\penalty0 611--630, 2004.

\bibitem[Komaki(2006)]{komaki2006}
F.~Komaki.
\newblock Shrinkage priors for {B}ayesian prediction.
\newblock \emph{{Annals of Statistics}}, 34:\penalty0 808--819, 2006.

\bibitem[Segers et~al.(2014)Segers, van~den Akker, and Werker]{segers2014}
J.~Segers, R.~van~den Akker, and B.~J.~M. Werker.
\newblock Semiparametric {G}aussian copula models: Geometry and efficient rank-based estimation.
\newblock \emph{The Annals of Statistics}, 42:\penalty0 1911--1940, 2014.

\bibitem[Sei and Komaki(2022)]{sei2022}
T.~Sei and F.~Komaki.
\newblock A correlation-shrinkage prior for {B}ayesian prediction of the two-dimensional {W}ishart model.
\newblock \emph{Biometrika}, 109:\penalty0 1173--1180, 2022.

\bibitem[Yang and Berger(1994)]{yang1994}
R.~Yang and J.~O. Berger.
\newblock Estimation of a covariance matrix using the reference prior.
\newblock \emph{{Annals of Statistics}}, 22:\penalty0 1195--1211, 1994.

\end{thebibliography}

\end{document}